\documentclass[11pt,a4paper]{amsart} 

\usepackage{amsmath}
\usepackage{amssymb}
\usepackage{amsthm}
\usepackage[initials,msc-links]{amsrefs}
\usepackage[dvipsnames]{xcolor}
\usepackage[english=american]{csquotes}
\usepackage{enumerate}
\usepackage{stmaryrd} 
\usepackage{mathrsfs}
\usepackage{mathtools}
\usepackage{bbm}
\usepackage{hyperref}
\hypersetup{
bookmarks=true,         
unicode=false,          
pdftoolbar=true,        
pdfmenubar=true,        
pdffitwindow=false,     
pdfstartview={FitH},    
pdftitle={My title},    
pdfauthor={Author},     
pdfsubject={Subject},   
pdfcreator={Creator},   
pdfproducer={Producer}, 
pdfkeywords={keyword1, key2, key3}, 
pdfnewwindow=true,      
colorlinks=true,       
linkcolor=blue ,          
citecolor=blue ,        
filecolor=magenta,      
urlcolor=Aquamarine           
}

\newtheorem{theorem}{Theorem}
\newtheorem{lemma}[theorem]{Lemma}

\newtheorem{corollary}[theorem]{Corollary}

\theoremstyle{definition}

\newtheorem{remark}[theorem]{Remark}

\newcommand{\Crm}{\mathrm{C}}

\newcommand{\Lrm}{\mathrm{L}}

\newcommand{\Dcal}{\mathcal{D}}

\newcommand{\Hcal}{\mathcal{H}}
\newcommand{\Ical}{\mathcal{I}}

\newcommand{\Nbf}{\mathbf{N}}

\newcommand{\Nbb}{\mathbb{N}}

\newcommand{\Pbb}{\mathbb{P}}

\DeclareMathOperator*{\wslim}{w*-lim}

\DeclareMathOperator{\curl}{curl}

\DeclareMathOperator{\rank}{rank}

\DeclareMathOperator{\spn}{span}

\newcommand{\set}[2]{\left\{\, #1 \  \textup{\textbf{:}}\  #2 \,\right\}}

\newcommand{\setb}[2]{\bigl\{\, #1 \  \textup{\textbf{:}}\  #2 \,\bigr\}}

\newcommand{\setBB}[2]{\biggl\{\, #1 \  \textup{\textbf{:}}\  #2 \,\biggr\}}

\newcommand{\dpr}[1]{\langle #1 \rangle}

\newcommand{\cl}[1]{\overline{#1}}

\newcommand{\dd}{\;\mathrm{d}}

\newcommand{\R}{\mathbb{R}}

\newcommand{\loc}{\mathrm{loc}}
\newcommand{\sym}{\mathrm{sym}}

\newcommand{\toweakstar}{\overset{*}\rightharpoonup}

\newcommand{\todown}{\downarrow}

\newcommand{\BV}{\mathrm{BV}}

\newcommand{\eps}{\epsilon}

\DeclareMathOperator{\Tan}{Tan}

\newcommand{\Leb}{\mathscr L}

\renewcommand{\eps}{\varepsilon}
\renewcommand{\phi}{\varphi}

\newcommand{\M}{\mathcal M}

\makeatletter
\renewcommand*\env@matrix[1][*\c@MaxMatrixCols c]{%
  \hskip -\arraycolsep
  \let\@ifnextchar\new@ifnextchar
  \array{#1}}
\makeatother

\DeclareMathOperator{\Gr}{Gr}

\DeclareMathOperator{\A}{\mathcal A}

\DeclareMathOperator{\Div}{div}

\title[]{An elementary approach to \\ the dimension of measures satisfying \\ a 
first-order linear PDE constraint} 
\author[A.~Arroyo-Rabasa]{Adolfo Arroyo-Rabasa}
\date{\today}
\address{Mathematics Institute, The University of Warwick}
\email{\href{mailto:adolfo.arroyo-rabasa@warwick.ac.uk}{adolfo.arroyo-rabasa@warwick.ac.uk}}
\thanks{This project has received funding from the European Research Council (ERC) under the European Union's Horizon 2020 research and innovation programme, grant agreement No 757254 (SINGULARITY)}

\begin{document}
\maketitle


\begin{abstract}
	We give a simple criterion on the set of probability tangent measures $\Tan(\mu,x)$ of a positive Radon measure $\mu$, which yields lower bounds on the Hausdorff dimension of $\mu$. As an application, we give an \emph{elementary} and purely algebraic proof of the sharp Hausdorff dimension lower bounds for first-order linear PDE-constrained measures;  bounds for closed (measure) differential forms and normal currents are further discussed. A \emph{weak} structure theorem in the spirit of [Ann. Math. 184(3) (2016), pp. 1017--1039] is also discussed for such measures. 
\end{abstract}


\section{Introduction}

The question of determining the dimension of a vector-valued Radon measure satisfying a PDE-constraint is a longstanding one. A good starting point are $\curl$-free measure fields. The seminal work of \textsc{De Giorgi}~\cite{de-giorgi1961frontiere-orien} on the structure of \emph{sets of finite perimeter} and the \emph{co-area formula}~\cite{fleming1960an-integral-for} from \textsc{Fleming \& Rishel} yield the estimate $|Du| \ll \Hcal^{d-1}$ for all distributional gradients $Du$ represented by a Radon measure. Later on, \textsc{Federer}  extended (see~\cite[Sec.~4.1.21]{federer1969geometric-measu}) this result to the estimate $\|T\| \ll \Ical^m \ll \Hcal^m$ for $m$-dimensional normal currents $T \in \Nbf_{m}(\R^d)$.\footnote{Here, $\Ical^m$ is the $m$-integral-geometric measure on $\R^d$.} 
Recently, these results have been further extended to deal with more general differential constraints (in the context of $\A$-free measures). Namely, in~\cite{arroyo-rabasa2018dimensional-est}  it is shown that $|\mu| \ll \Ical^{\ell_{\Pbb^k}} \ll \Hcal^{\ell_{\Pbb^k}}$ for measures $\mu$ satisfying a generic constraint of the form $P(D) \mu = 0$, where $P(D)$ is a $k$th-order linear partial differential operator with constant coefficients and $\ell_{\Pbb^k}$ is a positive integer depending only on the principal symbol ${\Pbb^k}$ of $P(D)$. This $\ell_{\Pbb^k}$ dimensional estimate turns out to be sharp for first-order operators; for higher-order operators it is an open question whether it remains an optimal bound (see ~\cite[Conjecture~1.6]{arroyo-rabasa2018dimensional-est}). 

The compendium of results mentioned above are of  stronger structural character than the ones presented on this note, since only bounds on the \emph{Hausdorff dimension} of such measures will be discussed here. However, they also require a significantly stronger machinery. 
Our main interest is to give a self-contained and \enquote{elementary} proof of the Hausdorff dimension (sharp) bounds for measures $\mu \in \M(\Omega, E)$ solving, in the sense of the distributions, an equation of the form 
\begin{equation}\label{op}
P(D)\mu \coloneqq \sum_{i = 1}^d P_i [\partial_i \mu]  + P_0 \mu = 0\,, \quad P_0,P_i \in F \otimes E\,,
\end{equation}
where $E,F$ are finite dimensional euclidean spaces. 

The angular stone of our proof rests on a rather simple \emph{invariance criterion} 
affecting all normalized blow-ups of a given positive Radon measure $\sigma$, which effortlessly yields a lower bound on the {Hausdorff dimension} $\dim_\Hcal(\sigma)$, where as usual
\[
\dim_\Hcal(\sigma) \coloneqq \sup\setb{0 \le \kappa \le d }{\sigma \ll \Hcal^\kappa}\,.
\]
This criterion (contained in Lemma~\ref{thm:dim}) links the \emph{vector-space dimension}, of those directions with respect to which a blow-up of $\sigma$ may be an invariant measure, to a lower bound of the {Hausdorff dimension}. 
In particular, this re-directs the study of dimensional estimates for  measures satsifying~\eqref{op}, to the study of the structural rigidity of their sets $\Tan(|\mu|,x)$ of \emph{probability tangent measures} (described in Sec.~\ref{prel}). (A similar method for establishing dimensional estimates has been considered in~\cite{ambrosio1997a-measure-theor} by \textsc{Ambrosio \& Soner}; see also~\cite{fragala1999on-some-notions} for the slightly more restrictive context of \emph{tangent spaces} $T_\sigma(x) \subset \R^d$  introduced by \textsc{Bouchitt\'e}, \textsc{Buttazzo} and \textsc{Seppecher}.) 

The advantage of this viewpoint lies in the fact that the principal symbol
\[
\xi \mapsto \Pbb(\xi) \coloneqq \sum_{i = 1}^d \xi_iP_i, \quad \xi \in \R^d,
\] 
being linear as function of $\xi$, precisely characterizes those directions where tangent measures are invariant measures. Thus, allowing one to define a dimension associated to the principal part of the operator:
\begin{equation}\label{eq:11}
\ell_{\Pbb} \coloneqq \min_{e \in E\setminus \{0\}} \,\dim\big(\,\{\Pbb [e] \equiv 0\}^\perp\,\big)\,.
\end{equation}
Here, we have used the short-hand notation $\{\Pbb [e] \equiv 0\} \coloneqq \set{\xi}{\Pbb(\xi)[e] = 0}$.
Note that this definition of dimension agrees with the definition given in~\cite[eq.~(1.6)]{arroyo-rabasa2018dimensional-est}. It may be worth to mention that, in the context of cocancelling operators (introduced by \textsc{Van Schaftingen} \cite{van-schaftingen2013limiting-sobole} and further extended in~\cite{arroyo-rabasa2018dimensional-est}; see also~\cite{raita2018l1-estimates-an,spector2018optimal-embeddi}), $P(D)$ is an $(\ell_{\Pbb} - 1)$-cocancelling operator. 

Our main result is contained in the following theorem:
	\begin{theorem}\label{thm:2}
	Let $\Omega \subset \R^d$ be an open set, let $P(D)$ be a first-order differential operator as in~\eqref{op}, and let $\mu \in \M(\Omega;E)$ be a solution of the equation
	\[
	P(D)\mu = 0 \quad \text{in the sense of distributions on $\Omega$}.
	\]
Then,
	\[
	\dim_{\Hcal}(|\mu|) \ge \ell_{\Pbb}.
	\]
Moreover, this  estimate is sharp since the measure
\[
\mu = e \, \Hcal^{\ell_\Pbb} \llcorner \{\Pbb[e] \equiv 0\}^\perp
\]
is a solution of~\eqref{op} on $\R^d$,
whenever $e \in E$ is any vector at which the minimum in~\eqref{eq:11} is attained.
\end{theorem}

\begin{remark}
The proof of Theorem~\ref{thm:2} does not require, in any way, the  \emph{structure theorem for PDE-constrained measures}~\cite[Theorem~1.1]{de-philippis2016on-the-structur}.
\end{remark}

At all points $x$ where $[P(D)\circ\frac{\dd \mu}{\dd |\mu|}(x)]$ is elliptic, that is, precisely when the polar $\frac{\dd \mu}{\dd |\mu|}(x)$ {does not} belong to the \emph{wave cone} set
\[
\Lambda_{\Pbb} \coloneqq \bigcup_{\xi \in \R^d} \ker \Pbb(\xi) \subset E,
\] 
the sets $\Tan(|\mu|,x)$ turn out to be trivial (containing only fully-invariant measures). The invariance criterion then allows us to give the following \emph{soft} version of~\cite[Theorem~1.1]{de-philippis2016on-the-structur} (see also~\cite{alberti1993rank-one-proper} in the case of gradients):

\begin{corollary}[weak structure theorem]\label{cor} Let $\Omega \subset \R^d$ be an open set, let $P(D)$ be a first-order differential operator as in~\eqref{op}, and let $\mu \in \M(\Omega;E)$ be a solution of the equation
	\[
	P(D)\mu = 0 \quad \text{in the sense of distributions on $\Omega$}.
	\]
Then, 
	\[
		\textstyle{|\mu| \llcorner \setb{x \in \Omega}{\frac{\dd \mu}{\dd |\mu|}(x) \notin \Lambda_{\Pbb}} \ll \Hcal^\kappa \quad \text{for all $0 \le \kappa < d$}.}
	\]
\end{corollary}

\begin{remark} The results contained in Theorem~\ref{thm:2} and Corollary~\ref{cor} apply to solutions of the inhomogeneous equation 
\[
P(D)\mu = \tau \in \M(\Omega;F).
\]
To see this, let $\tilde \mu = (\mu,\tau)$, $\tilde E = E \times F$, and consider the operator $\tilde P(D)\tilde \mu = P(D)\mu - \tau$.
\end{remark}

\subsection*{Further comments}


Both Theorem~\ref{thm:2} and Lemma~\ref{thm:dim} \emph{do not} lead to rectifiability, nor estimates of the form $|\mu| \ll \Ical^{\ell_\Pbb}$, or even $|\mu| \ll \Hcal^{\ell_\Pbb}$ by the methods presented on this note. This assertion is in line with the following observation. The shortcoming of Corollary~\ref{cor} ---with respect to the (strong) structure theorem--- lies in the requirement of $\kappa$ being strictly smaller than $d$.
As it has been remarked by \textsc{De Lellis} (see~\cite[Proposition~3.3]{de-lellis2008a-note-on-alber}), \textsc{Preiss'} example~\cite[Example~5.8(1)]{preiss1987geometry-of-mea} of a purely singular measure with only trivial \emph{tangent measures} hinders the hope for a traditional \emph{blow-up strategy} leading to the estimate in the critical case $\kappa = d$.\footnote{The definition of \emph{tangent measure} introduced by Preiss in~\cite{preiss1987geometry-of-mea} is slightly different than our definition of probability tangent measure. However, the same triviality in the cited example can be inferred for our notion of tangent measure (see~\cite[Remark~14.4(1)]{mattila1995geometry-of-set}).}
	
	In a forthcoming paper~\cite{arroyo-rabasa2018rigidity-of-tan}, it will be shown that all functions $u : \Omega \to \R^d$ of \emph{bounded deformation} satisfy the following rigidity property: every probability tangent measure $\tau \in \Tan(Eu,x)$ can be split as a sum of $1$-directional measures (here, $Eu = \frac{1}{2}(Du + Du^\mathrm{t}) \in \M(\Omega;\sym(\R^d \otimes \R^d))$ is the distributional symmetric gradient of $u$). Hence, by Lemma~\ref{thm:dim}, one may recover the dimensional estimate $\dim_\Hcal(|Eu|) \ge d-1$ from~\cite{ambrosio1997fine-properties} through a completely different method. Note however that symmetric gradients satisfy the St.~Venant compatibility conditions (see~\cite[Example~3.10(e)]{fonseca1999mathcal-a-quasi}) which is a 2nd-order differential constraint.

	\subsection*{Organization} Applications of our results for several relevant first-order operators are discussed in Section~\ref{sec:2}; dimension bounds for closed differential forms and normal currents are discussed in  Corollaries~\ref{cor:1}-\ref{cor:3}. A brief list of definitions (required for the proofs) and the invariance criterion (contained in Lemma~\ref{thm:dim}) are given in Section~\ref{prel}. Section~\ref{sec:proofs} is devoted to the proofs. Lastly, an appendix on multilinear algebra operations has been included, this may be of use for the applications on differential forms and normal currents discussed below.


\subsection*{Acknowledgments} I gratefully thank G.~de Philippis and F.~Rindler for introducing me to this problem, and to other related questions. I would also like to thank J. Hirsch and P. Gladbach for several fruitful discussions about this subject.


\section{Applications}\label{sec:2}

In this section we discuss explicit dimensional bounds for several relevant first-order differential operators. 

Here and in what follows $\Omega \subset \R^d$ is an open set.

\subsection{Gradients} The space $\BV(\Omega;\R^m)$ of functions of bounded variation  consists of functions $u : \Omega \to \R^m$ whose distributional derivative $Du$ can be represented by a Radon measure $\mu$ in $\M(\Omega;\R^m \otimes \R^d)$. We recall (see~\cite{fonseca1999mathcal-a-quasi}) that the gradient $\mu = Du$ is (locally) a $\curl$-free field in the sense that
\[
\curl (\mu) \coloneqq (\partial_i \mu_{kj} - \partial_j \mu_{ki})_{kij} = 0, \qquad 1 \le i,j \le d, \; 1 \le k \le m. 
\]
In the case $P(D) = \curl$ we have 
\[
\ker \Pbb_{\curl}(\xi) = \setb{a \otimes \xi}{a \in \R^m}, \quad \xi \in \R^d, 
\]
and therefore $\ell_{\curl} = d -1$. Theorem~\ref{thm:2} then recovers the well-known (see~\cite{ambrosio1997fine-properties}) dimensional bound for gradients 
\[
u \in \BV(\Omega;\R^m) \; \Longrightarrow \; \dim_{\Hcal}(|Du|) \ge d-1.
\]   

\subsection{Fields of bounded divergence} Consider the divergence operator defined on matrix-fields $\boldsymbol \mu \in \M(\Omega;\R^k \otimes \R^d)$ defined as
\begin{equation}\label{div}
\Div \boldsymbol \mu = \Bigg(\sum_{i = 1}^d \partial_i \mu_{ij}\Bigg)_j, \quad 1 \le j \le k.
\end{equation}
In this case we get $\Pbb_{\Div}(\xi)[M] = M \cdot \xi$ over the space of tensors $M \in \R^k \otimes \R^d$, and  $\{\Pbb_{\Div}[M] \equiv 0\}^\perp = (\ker M)^\perp \cong \mathrm{ran} \, M$. It follows from Riesz' representation theorem ($\frac{\dd \boldsymbol \mu}{\dd |\boldsymbol \mu|}(x) \neq 0$ for $|\boldsymbol \mu|$-a.e. $x$) and Theorem~\ref{thm:2} that
\[
\Div \boldsymbol \mu \in  \M(\Omega;\R^k) \quad \Longrightarrow \quad \dim_\Hcal(|\boldsymbol \mu|) \ge 1.
\] 
In a further refinement, we get the following corollary:

\begin{corollary}\label{prop2}Let $\boldsymbol \mu \in \M(\Omega;\R^k \otimes \R^d)$ satisfy the non-homogeneous equation $\Div \boldsymbol \mu = \tau$ for some $\tau \in \M(\Omega;\R^k)$.
Further, assume the set
\[
\setBB{x \in \Omega}{\rank \frac{\dd \boldsymbol\mu}{\dd |\boldsymbol\mu^s|}(x) \ge \ell}
\]
has full $|\boldsymbol\mu^s|$-measure on $\Omega$. Then, $\dim_{\Hcal}(|\boldsymbol\mu|) \ge \ell$.
\end{corollary}
\begin{proof}
In this case $\dim(\{\Pbb_{\Div}[M] \equiv 0\}^\perp) = \rank M  \ge \ell$. Then, by~\eqref{eq:final} and Lemma~\ref{thm:dim}, one gets the desired bound $\dim_{\Hcal}(|\boldsymbol\mu|) \ge \ell$.
\end{proof}

\subsection{Measure differential forms} Let $m \in \{0,\dots,d-1\}$ and let $\omega \in \M(\Omega;\bigwedge^m \R^d)$ be a \emph{measure $m$-form}. The \emph{exterior derivative} of $\omega$ 
is the $(m+1)$-form distribution 
\[
	\mathrm d\omega \coloneqq \sum_{\substack{i = 1,\dots,n\\1 \le i_1 < \dots < i_m \le n}} \partial_i \omega_{i_1\cdots i_m} [\dd x_i \wedge \dd x_{i_1} \wedge \cdots \wedge  \dd x_{i_m}],
	\]
where the $\omega_{i_1\cdots i_m} = \dpr{\omega,\dd x_{i_1} \wedge \cdots \wedge  \dd x_{i_m}} \in \M(\Omega)$ are the coefficients of $\omega$. The exterior derivative defines a first-order operator of the form~\eqref{op} with $V = \bigwedge^m \R^d$ and $F = \bigwedge^{m+1} \R^d$, and a principal symbol $\mathbbm d(\xi) : \bigwedge^m \R^d \to \bigwedge^{m+1} \R^d$ acting on $m$-co-vectors as
  	\[
  		\mathbbm d(\xi)[v^*] = \xi^* \wedge v^*.
  	\]
Here, $w^*\in \bigwedge^m \R^d$ is the image of $w \in \bigwedge_m \R^d$ under the canonical isomorphism.
By Lemma~\ref{lem:A} in the Appendix, we get $\{\mathbbm d[v^*] \equiv 0\} = \mathrm{Ann^1}(v^*)$ (see~\eqref{eq:ann^1} in the Appendix) and therefore $\ell_{\mathbbm d} = d - m$.
\begin{corollary}\label{cor:1}Let $\omega \in \M(\Omega;\bigwedge^m)$ be a measure $m$-form satisfying $\mathrm d\omega = \eta$ for some  $\eta \in \M(\Omega;\bigwedge^{m+1}\R^d)$. Then, $\omega$ satisfies the dimensional estimate
	\[
	\dim_{\Hcal}(|\omega|) \ge d - m.
	\]
\end{corollary}

\subsection{Normal currents}
Let $1 \le m \le d$ be an integer. The space of $m$-\emph{currents} consists of all distributions $T \in \Dcal'(\Omega;\bigwedge_m\R^d)$. In duality with the space of smooth differential forms and the exterior derivative, one defines the \emph{boundary} of a current $T$ as the $(m-1)$-current acting on $\Crm^\infty_c(\Omega;\bigwedge^{m-1}\R^d)$ as $\partial T [\omega] = T(\mathrm d\omega)$. 
The space $\Nbf_m(\Omega)$ of \emph{$m$-dimensional normal currents} is defined as the space of $m$-currents $T$, such that both $T$ and $\partial T$ can be represented by a measure, that is,
\[
		\textstyle{	\Nbf_m(\Omega) \cong \setb{T \in \M(\Omega;\bigwedge_m\R^d))}{\partial T \in \M(\Omega;\bigwedge_{m-1}\R^d)} }.
\]  
The total variation of a normal current $T$ is denoted by $\|T\|$; and we write $T = \vec T \,\|T\|$ to denote its polar decomposition.
The boundary operator on $\Nbf_m(\Omega)$ defines a first-order operator of the form~\eqref{op}, with a principal symbol $\mathbbm d^*(\xi) : \bigwedge_m\R^d \to \bigwedge_{m-1}\R^d$ acting on $m$-vectors as the interior multiplication
\[
	\mathbbm d^*(\xi)[v] = v \llcorner \xi^* \quad \text{where} \,\; \dpr{v \llcorner \xi^*,z^*} = \dpr{v,\xi^* \wedge z^*}.
\]
Using the notation contained in the appendix, we readily check that $\{\mathbbm d^*[v] \equiv 0\} = \mathrm{Ann_1}(v)$. By means of Lemma~\ref{lem:B} and definition~\eqref{eq:11}, we conclude $\ell_{\mathbbm d^*} = m$. Theorem~\ref{thm:2} gives an alternative proof of the known dimensional estimates for normal currents: 
\begin{corollary}
	Let $T = \vec{T}\|T\| \in \Nbf_m(\Omega)$ be an $m$-dimensional normal current on $\Omega$. Then, $\|T\|$ satisfies the dimensional estimate
		\[
			\dim_{\Hcal}(\|T\|) \ge m.
		\] 
\end{corollary}
Moreover, by the natural association between fields with bounded divergence and one-dimensional normal currents, Corollary~\ref{cor} and Proposition~\ref{prop2} yield a simple proof of the following soft version of~\cite[Corollary~1.12]{de-philippis2016on-the-structur}: 

\begin{corollary}\label{cor:3} Let $T_1 = \vec T_1\|T_1\|, \dots, T_d = \vec T_d\|T_d\| \in \Nbf_1(\Omega)$ be one-dimensional normal currents and assume there exists a positive Radon measure $\sigma \in \M(\Omega)$ satisfying the following properties:
\begin{enumerate}
	\item[\textnormal{(i)}] $\sigma \ll \|T_i\|$ for all $i = 1,\dots,d$,
	\item[\textnormal{(ii)}] $\spn\big\{\vec{T_1}(x),\dots,\vec{T_d}(x)\big\} = \R^d$ for $\sigma$-almost every $x \in \R^d$.
\end{enumerate}
 Then, $\sigma \ll \Hcal^\kappa$ for all $0 \le \kappa < d$.
\end{corollary}

\section{Preliminaries}\label{prel}
  
Let $E$ be a finite dimensional euclidean space. We denote by $\M(\Omega;E) \cong \Crm_c(\Omega;E)^*$ the space of $E$-valued Radon measures over $\Omega$. 
For a vector-valued measure $\mu \in \M(\Omega;E)$, we write the Radon--Nykod\'ym--Lebesgue decomposition of $\mu$ as
\[
\mu = \mu^{ac} \Leb^d \llcorner \Omega + g_\mu |\mu^s|, \qquad |g_\mu|_E = 1,
\]
where $\mu^{ac} \in \Lrm^1(\Omega;E)$, $|\mu^s| \perp \Leb^d \llcorner \Omega$, and $g_\mu \in \Lrm^1(\Omega,|\mu^s|;E)$. 

The map $T^{r,x}(y) = ({y - x})/{r}$,
which maps the open ball $B_r(x) \subset \R^d$ into the open unit ball $B_1 \subset \R^d$, induces a (isometry) push-forward $T^{r,x}_\# : \M(\R^d;E) \to \M(\R^d;E)$. A (normalized) sequence of the form 
\[
\gamma_j = \frac{1}{|\mu|(B_{r_j}(x))} \, (T^{r_j,x_0}_\# \mu) \llcorner B_1, \qquad r_j \todown 0, \quad j \in \Nbb,
\]
is called a \emph{bounded blow-up} sequence of $\mu$ at $x_0$. 
If $\tau = \wslim \gamma_j$ on $\M(\cl {B_1})$, we say that $\sigma$ is a \emph{probability tangent measure} of $\mu$ at $x_0$, symbolically we denote this by
\[
\tau \in \Tan(\mu,x_0).
\]
%
%
Observe that $|\tau|(\cl{B_1}) = 1$ and, at a $|\mu|$-Lebesgue point $x_0 \in \Omega$, it holds
\[
\tau \in \Tan(\mu,x_0) \quad \Longleftrightarrow \quad \tau = \frac{\dd \mu}{\dd |\mu|}(x_0) |\tau| \quad \text{and} \quad |\sigma|\in\Tan(|\mu|,x_0).
\]
For this an other facts about $\Tan(\mu,x)$, we refer the interested reader to the monograph~\cite[Sec.~{2.7}]{ambrosio2000functions-of-bo}.

For a finite dimensional euclidean vector space $W$, we write $\mathrm{Gr}(W)$ to denote the \emph{Grassmanian} of all linear subspaces of $W$, and $\Gr(\ell,W)$ to denote the set of $\ell$-dimensional subspaces of $W$; when $W = \R^d$ we shall simply write $\mathrm{Gr}(d)$ and $\mathrm{Gr}(\ell,d)$ respectively. 
For given $V \in \mathrm{Gr}(d)$, a measure $\mu \in \M(\R^d)$ is called \emph{$V$-invariant} if 
\[
\tau_\#\mu = \mu \quad \text{for all translations $\tau : \R^d \to \R^d$ satisfying $\tau(V) = V$}.
\] 
The subspace of $V$-invariant measures is denoted $\M^V(\R^d)$. 
Note that this space is sequentially weak-$*$ closed in $\M(\R^d)$. 

The dimension criterion is contained in the next result:

\begin{lemma}[invariance criterion]\label{thm:dim} Let $0 \le \ell \le d$ be a positive integer and let  $\sigma \in \M(\Omega)$ be a positive measure. Assume that at, $\sigma^s$-almost every $x \in \Omega$, every bounded tangent measure $\tau \in \Tan(\sigma^s,x)$ can be split on $B_1$ as a finite sum 
	\begin{equation}\label{eq:C}
	\tau= (\tau_1 + \dots + \tau_{k})\llcorner {B_1}, \qquad k = k(\sigma) \in \Nbb, \tag{C}
	\end{equation}
	where, for each $1 \le h \le k$, $\tau_h$ is a ${V_h}$-invariant measure for some $V_h \in \Gr(\ell_h,d)$ with $\ell_h \ge \ell$. Then, $\sigma$ satisfies the dimensional estimate 
	\[
	\dim_{\Hcal}(\sigma) \ge \ell. 
	\]  
\end{lemma}

\section{Proofs}\label{sec:proofs}

We begin by proving an estimate on the upper Hausdorff densities. 
 
\begin{lemma} Let $0 \le \ell \le d$ be a positive integer and let  $\sigma \in \M_\loc(\Omega)$ be a positive measure. Let $x \in \Omega$ be a $|\mu^s|$-Lebesgue point and assume that every bounded tangent measure $\tau \in \Tan(\mu^s,x)$ can be split on $B_1$ as a finite sum
	\[
	\tau = \tau_1 + \dots + \tau_{k}, \qquad k = k(\sigma) \in \Nbb,
	\]
	where each $\tau_h$ is a $V_h$-invariant measure for some $V_h \in \Gr(\ell_h,d)$ with $\ell \le \ell_h$.
	
	Then, the upper $\kappa$-density of $\mu$ at $x$ is equal to zero for all $\kappa \in [0,\ell)$, that is, 
	\[
	\theta^{*\kappa}(\mu^s,x) \coloneqq \limsup_{r \todown 0}\frac{\mu(Q_r(x))}{r^{\kappa}} = 0 \quad \forall \, \kappa \in [0,\ell).
	\]
\end{lemma}
\begin{proof}
	It suffices to show that $\theta^{*\kappa}(\mu^s,x)$ is finite for all $\kappa \in [0,\ell)$. The fact that $\theta^{*\kappa}(\mu^s,x)$ is equally zero will then follow from the next simple observation: if $\theta^{*\kappa_1}(\mu^s,x) > 0$, then $\theta^{*\kappa}(\mu^s,x) = \infty$ for all $\kappa \in (\kappa_1,\ell)$. 
	
	We argue by contradiction.  Assume that $\theta^{*\kappa}(\mu^s,x) = \infty$ for some $\kappa \in [0,\ell)$ and let $t \in (0,4^{-\frac{d}{\ell-\kappa}})$. Then, by~\cite[Proposition~2.42]{ambrosio2000functions-of-bo}, there exists a bounded tangent measure $\tau \in \Tan(\mu^s,x)$ with 
	$t^{\kappa} \le \tau(\cl{B_t}) \le \tau(B_1) \le 1$. On the other hand, by assumption, we may find a positive integer $k = k(\tau)$ such that
	\[
	\tau = \tau_1 + \dots + \tau_{k},
	\]
	where each $\tau_h$ is a positive $V_h$-directional measure, for all $1 \le h \le k$. 	
	Let us denote by $\mathbf p_h : \R^d \to V_h^\perp$ the canonical projection so that 
	\[
	\tau_h(F) \le \Leb^{\ell_h}((\mathbf 1 - \mathbf p_h) F) \cdot \tilde \tau_h(\mathbf p_h F) \quad F \subset B_1,
	\]
	where up to a linear isometry transformation we have $\tau_h = \Leb^{\ell_h} \otimes \tilde \tau_h$. 
	
	Next, we use that $4t < 1$ and that $\ell \le \ell_h$ (for all $1 \le h \le k$) to obtain the estimate
	\begin{align*}
	t^\kappa \le \tau(\cl{B_t}) & = \tau_1(\cl{B_t}) + \dots + \tau_{k}(\cl{B_t}) \\
	&  \le (2t)^{\ell_1}\tilde\tau_1(\mathbf p_1{B_\frac{1}{4}}) + \dots + (2t)^{\ell_{k}}\tilde\tau_{k}(\mathbf p_{k} {B_\frac{1}{4}})\\
		&  \le (2t)^{\ell_1}2^{-(d-\ell_1)}\tau_1(B_1) + \dots + (2t)^{\ell_{k}}2^{-(d-\ell_k)}\tau_k(B_1)\\
	&  \le 2^dt^\ell \tau(B_1) \le 2^d t^\ell.
	\end{align*}
	This chain of inequalities implies $2^{-\frac{d}{\ell - \kappa}} \le t$, which directly contradicts our choice of $t$.  
	This shows $\theta^{*\kappa}(\mu,x) < \infty$, as desired. 
\end{proof}

\begin{proof}[Proof of Lemma~\ref{thm:dim}]
	Fix an arbitrary $\kappa \in [0,\ell)$. By the previous lemma and the assumption we know that the set $\Theta^\kappa_0 \coloneqq \set{x \in \Omega}{\theta^{*\kappa}(\sigma,x) = 0}$ has full $|\sigma^s|$-measure on $\Omega$. Hence, $\sigma^s \llcorner \Theta^\kappa_0 = \sigma^s$. Moreover, for every $\eps > 0$, it holds $\theta^{*\kappa}(\sigma^s,x) \le \eps$ for all $x \in \Theta^\kappa_0$. Then, the upper-density criterion contained in~\cite[Theorem~2.56]{ambrosio2000functions-of-bo} holds and therefore
	\[
	\sigma^s \llcorner \Theta^\kappa_0 \le 2^\kappa \eps\, \Hcal^\kappa \llcorner \Theta^\kappa_0 \quad \text{for all $\eps > 0$}.
	\]  
	Letting $\eps \todown 0$ we deduce that $\sigma^s(F) = 0$ whenever $\Hcal^\kappa(F \cap \Theta^\kappa_0) < \infty$ for a Borel set $F \subset \Omega$. By the definition of Hausdorff dimension,  this implies $\dim_\Hcal(\sigma^s) \ge \kappa$. Since $\kappa \in [0,\ell)$ was chosen arbitrarily and $\dim_{\Hcal}(\sigma) = \dim(\sigma^s)$, we conclude that $\dim_\Hcal(\sigma) \ge \ell$.
\end{proof}

\begin{proof}[Proof of Theorem~\ref{thm:2}]
	Let $x \in \Omega$ be a $|\mu|^s$-Lebesgue point so that every probability tangent measure $\sigma \in \Tan(\mu^s,x)$ can be written as $\sigma = e|\sigma|$ with $e = \frac{\dd \mu}{\dd |\mu|}(x) \in E$ and $|\sigma| \in \Tan(|\mu|^s,x)$. 
	
	Fix $\sigma \in \Tan(\mu^s,x)$. Note that $P^1(D)\sigma = 0$ in the sense of distributions on $B_1$, where $P^1(D)$ is the principal part of $P(D)$. This follows from the scaling rule
	\[
P^1(D) [T^{r_j,x}_\# \mu] = - r_j \cdot P_0[T^{r_j,x}_\# \mu],
	\]
where the term in the right-hand side converges strongly to zero (in the sense of distributions) as $j \to \infty$.
We now use the fact that $B_1$ is a star-shaped domain to define smooth approximations of $\sigma$ on $B_1$ as follows. Fix $\delta > 0$ to be a small parameter and define $\sigma_\delta \coloneqq (T^{1 - \delta,0}_\#\sigma) \star \rho_\delta \in \Crm^\infty(B_1;E)$, where $\rho_\delta$ is a standard mollifier at scale $\delta$. In this way
	$\sigma_\delta \, \Leb^d \llcorner B_1 \toweakstar \sigma$ and $|\sigma_\delta| \, \Leb^d \llcorner B_1 \toweakstar |\sigma|$ as $\delta \todown 0$ on $B_1$. Observe that, for each $\delta > 0$, the measure $\sigma_\delta$ (which satisfies $\sigma_\delta = e|\sigma_\delta|$) solves (in the classical sense) the homogeneous equation
	\[
	P^1(D)\sigma_\delta = \text{$\sum_{i = 1}^n  P_i[e]\, (\partial_i |\sigma_\delta|)= 0$ on $B_1$}.
	\] 
	In symbolic language this reads $\Pbb(\nabla |\sigma_\delta|)[e] = 0$, or equivalently, in terms of the differential inclusion, 
	\[
	\nabla (|\sigma_\delta|) \in \{\Pbb[e] \equiv 0\}  \; \text{on $B_1$}. 
	\]
We deduce that $\nabla (|\sigma_\delta|)(x)[\xi] = 0$ for all $\xi \in \{\Pbb[e] \equiv 0\}^\perp$ and all $x \in B_1$. In particular, for every $\delta > 0$, the measure $|\sigma_\delta| \Leb^d \llcorner B_1$ is $\{\Pbb[e] \equiv 0\}^\perp$-invariant. Since the space of $\{\Pbb[e] \equiv 0\}^\perp$-invariant measures is sequentially weak-$*$ closed, we infer  that
	\begin{equation}\label{eq:final}
	\text{$|\sigma| \in \Tan(|\mu|^s,x)$ is a $\{\Pbb[e] \equiv 0\}^\perp$-invariant measure on $B_1$}.
	\end{equation}

	Finally, since $x$ was chosen to be an arbitrary $|\mu|^s$-Lebesgue point, $|\mu|$ satisfies~\eqref{eq:C} with $\ell = \ell_{\Pbb}$. We conclude, by Lemma~\ref{thm:dim}, that $\dim_\Hcal(|\mu|) \ge \ell_{\Pbb}$. 	
\end{proof}

\begin{proof}[Proof of Corollary~\ref{cor}]
	By the very definition of $\Lambda_\Pbb$, it follows that $ \{\Pbb[e] \equiv 0\} = \{0\}$ for all $e \notin \Lambda_\Pbb$. Let us write $S_{\Pbb,\mu} \coloneqq \setb{x \in \Omega}{\frac{\dd \mu}{\dd |\mu|}(x) \notin \Lambda_{\Pbb}}$. From~\eqref{eq:final}, it follows that 
	$|\mu| \llcorner S_{\Pbb,\mu}$ satisfies the assumptions of Lemma~\ref{thm:dim} with $\ell = d$. Therefore $\dim_{\Hcal}(|\mu| \llcorner S_{\Pbb,\mu}) = d$. The sought estimate is then an immediate consequence of the definition of Hausdorff dimension.  
\end{proof}

\appendix

\section{Multilinear algebra}

Let $V$ be a finite dimensional euclidean space. The exterior algebra $\bigwedge^* V$ is a graded algebra with the \enquote{$\wedge$} product. Specifically $\wedge : \bigwedge^p V \times \bigwedge^q V \to \bigwedge^{p+q} V : (\xi^*,\omega^*) \mapsto \xi^* \wedge \omega^*$. 

In the particular situation when $p = 1$, this is the multiplication by $1$-covectors. 
As such, we can define the annihilator of this map on a fixed $m$-vector by setting
\begin{equation}\label{eq:ann^1}
 \textstyle{\mathrm{Ann^1}(v^*) \coloneqq \setb{\xi \in V}{\xi^* \wedge v^* = 0}}.
\end{equation}

\begin{lemma}\label{lem:A} Let $V$ be an euclidean space of dimension $d$, let $m \in \{0,\dots,d\}$ be a positive integer, and let $v^* \in \bigwedge^m V$ be a non-zero $m$-covector. Then 
$$
 \textstyle{\mathrm{Ann^1}(v^*) \in \Gr(\ell,V) \quad \text{for some $0 \le \ell \le m$}}.
$$
Moreover, if $v^*$ is a simple $m$-covector, then $\ell = m$.
\end{lemma}
\begin{proof}
The assertion that $\mathrm{Ann_1}(v)$ is in fact a linear space follows immediately from the bi-linearity of the wedge product. 
Notice also that, on simple vectors $v^* = v_1^* \wedge \cdots \wedge v_m^*$, the result is immediate since then $\mathrm{Ann^1}(v^*) = \spn\{v_1,\dots,v_m\}$ (so that $\ell = m$ in this case). Any automorphism $\phi$ of $V$ lifts to an automorphism $\Phi$ on $\bigwedge^* V$ satisfying
\[
\Phi(v_1^* \wedge \cdots \wedge v_m^*) = \phi(v_1^*) \wedge \cdots \wedge \phi(v_m^*).
\]
Hence, once $v^* \in \bigwedge_m V$ is fixed, we may assume without loss of generality that
$\mathrm{Ann}^1(v^*) = \spn\{e_1,\dots,e_\ell\}$ for some $0 \le \ell \le d$, where $\{e_1,\dots,e_d\}$ is an orthonormal basis of $V$. Indeed, let $\{\xi_1,\dots,\xi_\ell\}$ be a normal basis of $\mathrm{Ann}^1(v)$ and let $\phi$ be the automorphism of $V$ satisfying $\phi(\xi_i^*) = e_i^*$ for all $1 \le i \le \ell$ and $\phi(w^*) = w^*$ for all $w^* \in \mathrm{Ann}^1(v^*)^\perp$.
Then,
\[
\Phi(\xi_i^* \wedge v^*) = e_i^* \wedge \phi(v^*). 
\]

Let us fix $i_0 \in \{1,\dots,\ell\}$ and observe that
\[
e_{i_0}^* \wedge v^* = \sum_{\substack{1 \le i_1 < \dots < i_m \le d\\i_1,\dots,i_m \neq i_0}} v_{i_1 \cdots i_m} (e_{i_0}^* \wedge e_{i_1}^* \wedge \cdots \wedge e^*_{i_m}),
\]
where $v_{i_1 \cdots i_m} = \dpr{v^*,e_{i_1} \wedge \cdots \wedge e_{i_m}}$. On the other hand, the set
\[
\setb{e_{i_0}^* \wedge e_{i_1}^* \wedge \cdots \wedge e_{i_m}^*}{1 \le i_1 < \dots < i_m \le d, i_1, \dots i_m \neq i_0}
\]
conforms a set of linearly independent $m$-covectors in $\bigwedge^{m+1} V$. Therefore, $e_i^* \wedge v^* = 0$ if and only if $e^*_{i_0} \wedge e^*_{i_1} \wedge \cdots \wedge e^*_{i_m} = 0$ for all $1 \le i_1 < \dots < i_m \le d$ such that $v_{i_1 \cdots i_m} \neq 0$. Since $1 \le i_0 \le \ell$ was chosen arbitrarily, this yields the set contention 
\[
\mathrm{Ann^1}(v^*) \subset \bigcap_{\substack{i_1,\dots,i_m \neq i_0\\v_{i_1 \cdots i_m} \neq 0}}\mathrm{Ann^1}(e_{i_1}^*\wedge \cdots \wedge e_{i_m}^*).
\]
By the first observation, on the dimension of annihilators of simple vectors, we conclude that $\dim[\mathrm{Ann^1}(v^*)] = \ell \le m$.
\end{proof}

By duality, the exterior product induces tan interior multiplication on the algebra of vectors $\bigwedge_* V$. This is a bilinear map $\llcorner : \bigwedge_{q} V \times \bigwedge^{p} V \mapsto \bigwedge_{q-p} V : (v,w^*) \mapsto v \llcorner w^* $, where $v \llcorner w$ acts on $(q-p)$-co-vectors $z^*$ as
\[
\dpr{v \llcorner w^* , z^*} = \dpr{v, w^* \wedge z^*}.
\] 
Similarly as before, when $p =1$, we may consider its corresponding annihilator 
	\[
		\textstyle{\mathrm{Ann}_{1}(v) \coloneqq \setb{\xi \in \R^d}{ v \llcorner \xi^* = 0}}.
	\]
A similar (dual) proof to the one of Lemma~\ref{lem:A} yields the following result:
\begin{lemma}\label{lem:B} Let $V$ be an euclidean space of dimension $d$, let $m \in \{0,\dots,d\}$, and let $v \in \bigwedge_mV$ be a non-zero $m$-vector. Then
		\[
			\mathrm{Ann}_1(v) \in \mathrm{Gr}(d - \ell,V)  \quad\text{for some $0 \le \ell \le m$}.
		\]
Furthermore, if $v$ is a simple $m$-vector, then $\ell = d - m$.
\end{lemma}

\bibliography{zotero}

\begin{bibdiv}
\begin{biblist}

\bib{alberti1993rank-one-proper}{article}{
      author={{Alberti}, G.},
       title={{Rank one property for derivatives of functions with bounded
  variation}},
        date={1993},
     journal={{Proc. Roy. Soc. Edinburgh Sect. A}},
      volume={123},
      number={2},
       pages={239\ndash 274},
      review={\MR{1215412}},
}

\bib{ambrosio1997fine-properties}{article}{
      author={Ambrosio, L.},
      author={Coscia, A.},
      author={Dal~Maso, G.},
       title={Fine properties of functions with bounded deformation},
        date={1997},
        ISSN={0003-9527},
     journal={Arch. Rational Mech. Anal.},
      volume={139},
      number={3},
       pages={201\ndash 238},
  url={https://0-doi-org.pugwash.lib.warwick.ac.uk/10.1007/s002050050051},
      review={\MR{1480240}},
}

\bib{ambrosio2000functions-of-bo}{book}{
      author={Ambrosio, L.},
      author={Fusco, N.},
      author={Pallara, D.},
       title={{Functions of bounded variation and free discontinuity
  problems}},
      series={Oxford Mathematical Monographs},
   publisher={The Clarendon Press, Oxford University Press, New York},
        date={2000},
      review={\MR{1857292}},
}

\bib{ambrosio1997a-measure-theor}{article}{
      author={Ambrosio, Luigi},
      author={Soner, Halil~Mete},
       title={A measure theoretic approach to higher codimension mean curvature
  flows},
    language={en},
        date={1997},
     journal={Annali della Scuola Normale Superiore di Pisa - Classe di
  Scienze},
      volume={Ser. 4, 25},
      number={1-2},
       pages={27\ndash 49},
         url={http://http://www.numdam.org/item/ASNSP_1997_4_25_1-2_27_0},
      review={\MR{1655508}},
}

\bib{arroyo-rabasa2018rigidity-of-tan}{unpublished}{
      author={Arroyo-Rabasa, Adolfo},
       title={Rigidity of tangent functions of bounded variation and its
  applications in the calculus of variations},
        note={In preparation},
}

\bib{arroyo-rabasa2018dimensional-est}{article}{
      author={Arroyo-Rabasa, Adolfo},
      author={{De~Philippis}, Guido},
      author={Hirsch, Jonas},
      author={Rindler, Filip},
       title={Dimensional estimates and rectifiability for measures satisfying
  linear pde constraints},
        date={2018},
     journal={ArXiv e-prints: 1811.01847},
      eprint={1811.01847},
         url={https://arxiv.org/pdf/1811.01847},
}

\bib{de-giorgi1961frontiere-orien}{book}{
      author={{De Giorgi}, E.},
       title={Frontiere orientate di misura minima},
      series={Seminario di Matematica della Scuola Normale Superiore di Pisa,
  1960-61},
   publisher={Editrice Tecnico Scientifica, Pisa},
        date={1961},
      review={\MR{0179651}},
}

\bib{de-lellis2008a-note-on-alber}{incollection}{
      author={De~Lellis, Camillo},
       title={A note on {A}lberti's rank-one theorem},
        date={2008},
   booktitle={Transport equations and multi-{D} hyperbolic conservation laws},
      series={Lect. Notes Unione Mat. Ital.},
      volume={5},
   publisher={Springer, Berlin},
       pages={61\ndash 74},
  url={https://0-mathscinet-ams-org.pugwash.lib.warwick.ac.uk/mathscinet-getitem?mr=2504174},
      review={\MR{2504174}},
}

\bib{de-philippis2016on-the-structur}{article}{
      author={{De Philippis}, G.},
      author={{Rindler}, F.},
       title={{On the structure of ${\mathcal A}$-free measures and
  applications}},
        date={2016},
        ISSN={0003-486X; 1939-8980/e},
     journal={{Ann. Math.}},
      volume={184},
      number={3},
       pages={1017\ndash 1039},
      review={\MR{3549629}},
}

\bib{federer1969geometric-measu}{book}{
      author={Federer, Herbert},
       title={Geometric measure theory},
      series={Die Grundlehren der mathematischen Wissenschaften, Band 153},
   publisher={Springer-Verlag New York Inc., New York},
        date={1969},
  url={https://0-mathscinet-ams-org.pugwash.lib.warwick.ac.uk/mathscinet-getitem?mr=0257325},
      review={\MR{0257325}},
}

\bib{fleming1960an-integral-for}{article}{
      author={Fleming, Wendell~H.},
      author={Rishel, Raymond},
       title={An integral formula for total gradient variation},
        date={1960},
        ISSN={0003-889X},
     journal={Arch. Math. (Basel)},
      volume={11},
       pages={218\ndash 222},
  url={https://0-mathscinet-ams-org.pugwash.lib.warwick.ac.uk/mathscinet-getitem?mr=0114892},
      review={\MR{0114892}},
}

\bib{fonseca1999mathcal-a-quasi}{article}{
      author={Fonseca, I.},
      author={M{\"{u}}ller, S.},
       title={{{$\mathcal A$}-quasiconvexity, lower semicontinuity, and {Y}oung
  measures}},
        date={1999},
        ISSN={0036-1410},
     journal={SIAM J. Math. Anal.},
      volume={30},
      number={6},
       pages={1355\ndash 1390},
         url={http://dx.doi.org/10.1137/S0036141098339885},
      review={\MR{1718306}},
}

\bib{fragala1999on-some-notions}{article}{
      author={Fragal{\`a}, Ilaria},
      author={Mantegazza, Carlo},
       title={On some notions of tangent space to a measure},
        date={1999},
        ISSN={0308-2105},
     journal={Proc. Roy. Soc. Edinburgh Sect. A},
      volume={129},
      number={2},
       pages={331\ndash 342},
  url={https://0-mathscinet-ams-org.pugwash.lib.warwick.ac.uk/mathscinet-getitem?mr=1686704},
      review={\MR{1686704}},
}

\bib{mattila1995geometry-of-set}{book}{
      author={Mattila, P.},
       title={{Geometry of sets and measures in {E}uclidean spaces}},
      series={Cambridge Studies in Advanced Mathematics},
   publisher={Cambridge University Press, Cambridge},
        date={1995},
      volume={44},
        ISBN={0-521-46576-1; 0-521-65595-1},
         url={http://dx.doi.org/10.1017/CBO9780511623813},
      review={\MR{1333890}},
}

\bib{preiss1987geometry-of-mea}{article}{
      author={Preiss, David},
       title={Geometry of measures in {${\bf R}^n$}: distribution,
  rectifiability, and densities},
        date={1987},
        ISSN={0003-486X},
     journal={Ann. of Math. (2)},
      volume={125},
      number={3},
       pages={537\ndash 643},
  url={https://0-mathscinet-ams-org.pugwash.lib.warwick.ac.uk/mathscinet-getitem?mr=890162},
      review={\MR{890162}},
}

\bib{raita2018l1-estimates-an}{techreport}{
      author={Raita, Bogdan},
       title={{$L^1$}-estimates and {$\mathbb A$}-weakly differentiable
  functions},
 institution={University of Oxford, Technical Report OxPDE-18/01},
        date={2018},
}

\bib{spector2018optimal-embeddi}{article}{
      author={Spector, Daniel},
      author={Schaftingen, Jean~Van},
       title={Optimal embeddings into lorentz spaces for some vector
  differential operators via gagliardo's lemma},
        date={2018},
     journal={ArXiv e-prints: 1811.02691},
      eprint={1811.02691},
         url={https://arxiv.org/pdf/1811.02691},
}

\bib{van-schaftingen2013limiting-sobole}{article}{
      author={Van~Schaftingen, Jean},
       title={Limiting {S}obolev inequalities for vector fields and canceling
  linear differential operators},
        date={2013},
        ISSN={1435-9855},
     journal={J. Eur. Math. Soc.},
      volume={15},
      number={3},
       pages={877\ndash 921},
  url={https://0-mathscinet-ams-org.pugwash.lib.warwick.ac.uk/mathscinet-getitem?mr=3085095},
      review={\MR{3085095}},
}

\end{biblist}
\end{bibdiv}
\bibliographystyle{plain}
\end{document}